\documentclass[draft]{amsart}
\usepackage{hyperref}
\hypersetup{nesting=true,debug=true,naturalnames=true}
\usepackage{graphicx,amssymb}

\usepackage[square,sort&compress,comma,numbers]{natbib}
\usepackage{nicefrac,xcolor,upref}

\usepackage{amscd,amsthm,amsmath,amssymb,amsfonts}

\usepackage{mathrsfs}

\newtheorem{theorem}{Theorem}[section]
\newtheorem{lemma}[theorem]{Lemma}
\newtheorem{proposition}[theorem]{Proposition}
\newtheorem{definition}[theorem]{Definition}

\numberwithin{equation}{section}

\def\Q{{\mathbb {Q}}}

\def\Z{{\mathbb Z}}

\def\house#1{\setbox1=\hbox{$\,#1\,$}%
\dimen1=\ht1 \advance\dimen1 by 2pt \dimen2=\dp1 \advance\dimen2 by 2pt
\setbox1=\hbox{\vrule height\dimen1 depth\dimen2\box1\vrule}%
\setbox1=\vbox{\hrule\box1}%
\advance\dimen1 by .4pt \ht1=\dimen1
\advance\dimen2 by .4pt \dp1=\dimen2 \box1\relax}

  \def\eps{{\varepsilon}}

\def\build#1_#2^#3{\mathrel{\mathop{\kern 0pt#1}\limits_{#2}^{#3}}}

\def\date {le\ {\the\day}\ \ifcase\month\or
janvier\or fevrier\or mars\or avril\or mai\or juin\or juillet\or
ao\^ut\or septembre\or octobre\or novembre\or
d\'ecembre\fi\ {\oldstyle\the\year}}

\font\fivegoth=eufm5 \font\sevengoth=eufm7 \font\tengoth=eufm10

\newfam\gothfam \scriptscriptfont\gothfam=\fivegoth
\textfont\gothfam=\tengoth \scriptfont\gothfam=\sevengoth

\def\hw{{\widehat w}} 
\def\hla{{\widehat \lambda}}

\def\vp{{\bf p}}
\def\vy{{\bf y}}
\def\ZZ{{\mathbb Z}}

\def\QQ{{\mathbb Q}}

\begin{document}

\title[On simultaneous rational approximation]{On simultaneous rational approximation
to a $p$-adic number and its integral powers, II}

\author{Dzmitry Badziahin}
\address{The University of Sydney, Camperdown 2006, NSW (Australia)}
\email{dzmitry.badziahin@sydney.edu.au}

\author{Yann Bugeaud}
\address{Universit\'e de Strasbourg, Math\'ematiques,
7, rue Ren\'e Descartes, 67084 Strasbourg  (France)}
\email{bugeaud@math.unistra.fr}

\author{Johannes Schleischitz}
\address{Middle East Technical University, Northern Cyprus Campus, Kalkanli, G\"uzelyurt (Turkey)}
\email{johannes@metu.edu.tr}

\begin{abstract}
Let $p$ be a prime number.
For a positive integer $n$ and a
real number $\xi$, let $\lambda_n (\xi)$ denote the supremum
of the real numbers $\lambda$ for which there are
infinitely many integer tuples $(x_0, x_1, \ldots , x_n)$ such that
$| x_0 \xi - x_1|_p,   \ldots , | x_0 \xi^n - x_n|_p$ are all
less than $X^{-\lambda - 1}$, where $X$ is the maximum
of $|x_0|, |x_1|, \ldots , |x_n|$.
We establish new results on the Hausdorff dimension
of the set of real numbers $\xi$ for which $\lambda_n (\xi)$
is equal to (or greater than or equal to) a given value.
\end{abstract}


\maketitle

\section{Introduction}

Throughout, we let $p$ denote a prime number.
The present paper is concerned with the simultaneous approximation to successive integral powers of a given
transcendental $p$-adic number $\xi$ by rational numbers with the same denominator.
It is partly motivated by a mistake found in \cite{BBDO}:\footnote{The
ordering of the names of the authors of \cite{BBDO} should have been alphabetical; the publisher put
the corresponding author first, and this was overlooked during proofreading.}
namely,
the definition of the quantity $\lambda_n (\xi)$ given in \cite[Definition 1.2]{BBDO} is not accurate
and leads to the trivial result that $\lambda_n (\xi)$ is always infinite. Indeed, for any fixed nonzero
integer $y$
with $|y\xi|_{p}\leq 1$,
the quantity $|y \xi - x|_p$ can be made arbitrarily small, by taking a suitable
integer $x$. The authors overlooked the fact that, unlike in the real case,
we can have simultaneously $|y \xi - x|_p$ very small and $|x|$ much larger than $|y|$.
Fortunately, they (essentially) argue as they were using
the correct definition of the exponent of approximation $\lambda_n$, which is the following one.

\begin{definition}
\label{Def:1.1}
Let $n \ge 1$ be an integer and $\xi$ a $p$-adic number.
We denote by $\lambda_n(\xi)$
the supremum  of the real numbers
$\lambda$ such that, for
arbitrarily large real numbers $X$, the inequalities
$$
\max\{|x_0|, \ldots , |x_n|\} \le X, \qquad 0<\max_{1 \le m \le n} \,
|x_0 \xi^m - x_m|_p \le X^{-\lambda - 1},
$$
have a solution in integers $x_0, \ldots, x_n$.
\end{definition}

We will sometimes use that the above definition remains unchanged    
if we impose that the integers $x_0, \ldots, x_n$ have no common factor. 
To see this, it is sufficient to observe that, for every integers $x_0, \ldots, x_n$ 
and any positive integer $k$, if 
we have $|k x_0 \xi^m - k x_m|_p \le (k\max |x_{i}|)^{-\lambda-1}$, 
for $m = 1, \ldots , n$, then also $|x_0 \xi^m - x_m|_p \le (\max |x_{i}|)^{-\lambda-1}$, 
for $m = 1, \ldots , n$. This requires a short calculation.
The $p$-adic version of the Dirichlet theorem (see e.g. \cite{Mah34}) implies that
$\lambda_n (\xi) \ge 1/n$ for every irrational $p$-adic number $\xi$.
Furthermore, it follows from the $p$-adic Schmidt Subspace Theorem that
$\lambda_n (\xi) = \max \{1/n, 1/(d-1)\}$ for every positive integer $n$ and every
$p$-adic algebraic number $\xi$ of degree $d\geq 2$.
Moreover, almost every $p$-adic number $\xi$ satisfies $\lambda_n (\xi) = 1/n$ for every $n \ge 1$.
This follows from a classical result of Sprind\v zuk \cite{Spr69} combined
with the $p$-adic analogue of the Khintchine transference theorem, established by Jarn\'\i k \cite{Jar39}. 

In Definition \ref{Def:1.1}, the inequalities $|x_0|, \ldots , |x_n| \le X$ replace the inequalities
$0 < |x_0| \le X$ occurring in \cite[Definition 1.2]{BBDO}.
All results of \cite{BBDO} hold for the exponents $\lambda_n$ as in Definition \ref{Def:1.1}, but some of the
proofs have to be modified accordingly. Beside pointing out this mistake, we take the opportunity
to considerably extend \cite[Theorem 2.3]{BBDO} and \cite[Lemma 5.1]{BBDO}.
Our main motivation is the determination of the spectrum of $\lambda_n$, that is,
the set of values taken by $\lambda_n$
evaluated at transcendental $p$-adic numbers.
We also aim at extending to $\lambda_n$ the $p$-adic
Jarn\'\i k--Besicovich theorem \cite{Jar45,Lutz55}, which asserts that, for any $\lambda \ge 1$, we have
\begin{equation}
\label{1.1}
\dim \{ \xi \in \Q_p : \lambda_1 (\xi) \ge \lambda \} =
\dim \{ \xi \in \Q_p : \lambda_1 (\xi) = \lambda \} =
{2 \over 1 + \lambda}.
\end{equation}
Here and below, $\dim$
denotes the Hausdorff dimension.
We do not solve completely these problems, however, we
manage to establish the $p$-adic analogues of the results of \cite{Schl16} and \cite{BaBu20}.
In the course of the proofs, we obtain new transference theorems.
We also briefly investigate uniform simultaneous approximation and establish the $p$-adic analogues
of results of Laurent \cite{Lau03} and Davenport and Schmidt \cite{DaSc69}.
As a consequence, we can slightly improve a claim of Teuli\'e~\cite{Teu02} 
on approximation to a $p$-adic number
by $p$-adic algebraic numbers (resp., integers).

Throughout this paper,
$\lfloor \cdot \rfloor$ denotes the integer part function
and $\lceil \cdot \rceil$ the ceiling function.
The notation $a \gg_d b$ means that $a$ exceeds $b$ times a constant
depending only on $d$. When $\gg$ is written
without any subscript, it means that the constant is absolute.
We write $a \asymp b$ if both
$a \gg b$ and $a \ll b$ hold.

\section{Main results}

Our first result is a $p$-adic analogue of \cite[Corollary 1.8]{Schl16}.

\begin{theorem}
\label{Th:2.0}
Let $n \ge 2$ be an integer and $\lambda > 1$ a real number. Then, we have
$$
\dim \{ \xi \in \Q_p : \lambda_n (\xi) \ge \lambda\} =
\dim \{ \xi \in \Q_p : \lambda_n (\xi) = \lambda\} =
{2 \over n(1 + \lambda)}.
$$
\end{theorem}

Theorem \ref{Th:2.0} was established under the much weaker condition $\lambda > n- 1$ in \cite{BBDO}
(note that the correct definition of $\lambda_n$ is used in the proof of \cite[Theorem 2.3]{BBDO}).
By adapting the arguments of \cite{BaBu20} to the $p$-adic setting,
if $n\geq 3$ we can relax the assumption $\lambda > 1$ in Theorem \ref{Th:2.0}.
The next results are the $p$-adic analogues of
Theorems 2.1 to 2.3
of \cite{BaBu20}.

\begin{theorem}
\label{Th:2.1} Let $n \ge 3$ be an integer. The spectrum of $\lambda_n$
contains the interval $[(n+4)/ (3n), + \infty]$. Let $\lambda \ge (n+4) /
(3n)$ be a real number. Then, we have
$$
\dim \{ \xi \in \Q_p : \lambda_n (\xi) = \lambda\} =
{2 \over n(1 + \lambda)}.
$$
In particular, for any real number $\lambda$ with $\lambda > 1/3$, there exists an integer $n_0$
such that
$$
\dim \{ \xi \in \Q_p : \lambda_n (\xi) = \lambda\} =
{2 \over n(1 + \lambda)},
$$
for any integer $n$ greater than $n_0$.
\end{theorem}

We believe that Theorem~\ref{Th:2.1} holds for $n=2$ too. In view of
Theorem~\ref{Th:2.0}, the conclusion of this theorem is satisfied for all
$\lambda> 1 = (2+4)/(3\cdot 2)$. Therefore only the case of $n=2,\lambda=1$
remains open.

Theorem \ref{Th:2.2} below shows that the assumption `$\lambda > 1/3$' in the last
assertion of Theorem \ref{Th:2.1} is sharp.

\begin{theorem}
\label{Th:2.2}
For any integer $n \ge 2$, we have
$$
\dim \{ \xi \in \Q_p : \lambda_n (\xi) \ge 1/3 \} \ge {2 \over (n - 1) (1 + 1/3)}.
$$
\end{theorem}

Theorems \ref{Th:2.1} and \ref{Th:2.2} show that there is a discontinuity at $1/3$ in the following sense.
The function $n \mapsto n \dim \{ \xi \in \Q_p : \lambda_n (\xi) = \lambda\}$ is ultimately constant equal
to $2 / (1 + \lambda)$ for $\lambda > 1/3$, while the values taken by
the function
$$
n \mapsto n \dim \{ \xi \in \Q_p : \lambda_n (\xi) = 1/3\}
$$
are all greater than $2 / ( 1 + 1/3)$.

Theorems \ref{Th:2.1} and \ref{Th:2.2} above are special cases of the following general statement.

\begin{theorem}
\label{Th:2.3}
Let $k, n$ be integers with $1 \le k \le n$.
Let $\lambda $ be a real number with $\lambda \ge 1/n$.
Then we have
\begin{equation}
\label{2.1}
\dim\{ \xi \in \Q_p : \lambda_n (\xi) \ge \lambda \} \ge
{(k+1) (1 -  (k-1) \lambda ) \over (n - k + 1) (1 + \lambda )}.
\end{equation}
If $\lambda > 1 / \lfloor {n+1 \over 2} \rfloor$, then, setting $m = 1 +
\lfloor 1/ \lambda \rfloor$, we have
\begin{equation}
\label{2.2}
\dim \{ \xi \in \Q_p : \lambda_n (\xi)  \ge \lambda\}  \le
\max_{1\le h\le m} \left\{{(h+1)(1-(h-1)\lambda) \over (n-2h+2)(1+\lambda)} \right\}.
\end{equation}
\end{theorem}

Theorem \ref{Th:2.2} corresponds to \eqref{2.1} applied with $\lambda = 1/3$ and $k = 2$.
The deduction of Theorem~\ref{Th:2.1} from Theorem~\ref{Th:2.3} follows the same lines as in
the real case. Thus we omit it and refer the reader to \cite{BaBu20},
and only remark that the special case $n=2, \lambda=1$ remains open, since 
the $p$-adic analogue of the metric results on planar curves  
obtained in~\cite{bdv,vv} has not yet been established.
For $k \ge 2$, 
we remark that \eqref{2.1} is of interest for $\lambda<1/(k-1)$ only,
otherwise the right hand side is not positive.

A key ingredient for the proof of \eqref{2.1} is a transference inequality, which relates the
exponents of Diophantine
approximation $\lambda_n$ and $w_k$, with $k \le n$. The exponents $w_n$
were introduced by Mahler \cite{Mah35} to
measure how small an integer linear form in the first $n$
powers of a given $p$-adic number can be.

\begin{definition}
\label{Def:2.1}
Let $n \ge 1$ be an integer and $\xi$ a $p$-adic number.
We denote by $w_n(\xi)$
the supremum of the real numbers $w$ such that, for
arbitrarily large real numbers $X$, the inequalities
$$
0 < |x_n \xi^n + \ldots + x_1 \xi + x_0|_p
\le X^{-w-1}, \quad  \max_{0 \le m \le n} \, |x_m| \le X,
$$
have a solution in integers $x_0, \ldots, x_n$.
\end{definition}

For similar reasons as in Definition~\ref{Def:1.1} we can impose 
to the integers $x_0, \ldots , x_n$ in Definition~\ref{Def:2.1} to be coprime. 
It follows from the Dirichlet Box Principle that $w_n(\xi) \ge n$ for every $n \ge 1$ and every
transcendental $p$-adic number $\xi$.
For an algebraic number $\xi$ of degree $d$
we have $w_{n}(\xi)=\min\{n,d-1\}$. Finally we remark that
$w_{1}(\xi)=\lambda_{1}(\xi)$ for every $\xi$.

\begin{theorem}
\label{Th:3.1}
Let $k, n$ be integers with $1 \le k \le n$.
Let $\xi$ be a $p$-adic transcendental number. Then, we have
\begin{equation}
\label{3.1}
\lambda_n (\xi) \ge {w_k  (\xi) - n + k \over (k-1) w_k  (\xi) + n}.
\end{equation}
\end{theorem}

The case $k=n$ of Theorem \ref{Th:3.1}
has been established by Jarn\'\i k \cite{Jar39}   
(this is the $p$-adic analogue of an inequality in Khintchine's transference theorem).   

For completeness, we state another transference inequality,
whose real analogue is \cite[Theorem 2.4]{BaBu20}.

\begin{theorem}
\label{Th:2.4}
Let $\xi$ be a $p$-adic transcendental number.
For any positive integer $k$, we have
$$
(k+1) \bigl( 1 + \lambda_{k+1} (\xi) \bigr) \ge k \bigl( 1 + \lambda_k (\xi) \bigr),
$$
with equality if $\lambda_{k+1} (\xi) > 1$.
Consequently, for every integer $n$ with $n \ge k$, we have
$$
\lambda_n (\xi) \ge {k \lambda_k (\xi) - n + k \over n},
$$
with equality if $\lambda_n (\xi) > 1$.
\end{theorem}

Taking $k=1$ and $n$ arbitrary, Theorem \ref{Th:2.4} shows that, if $\lambda_n (\xi) > 1$, then
$\lambda_1 (\xi) = n \lambda_n (\xi) + n - 1$. Combined with \eqref{1.1}, this gives an
alternative proof of Theorem \ref{Th:2.0}.

We take the opportunity of this paper to present new results on the
exponents of uniform approximation $\hla_n$ and $\hw_n$.

\begin{definition}
\label{Def:1.1unif}
Let $n \ge 1$ be an integer and $\xi$ a $p$-adic number.
We denote by $\hla_n(\xi)$
the supremum  of the real numbers
$\hla$ such that, for
every sufficiently large real number $X$, the inequalities
$$
\max\{|x_0|, \ldots , |x_n|\} \le X, \quad 0<\max_{1 \le m \le n} \,
|x_0 \xi^m - x_m|_p \le X^{-\hla - 1},
$$
have a solution in integers $x_0, \ldots, x_n$.
We denote by $\hw_n(\xi)$
the supremum of the real numbers $\hw$ such that, for
every sufficiently large real number $X$, the inequalities
$$
\max\{|x_0|, \ldots , |x_n|\} \le X,  \quad 0<|x_n \xi^n + \ldots + x_1 \xi + x_0|_p  \le X^{- \hw-1},
$$
have a solution in integers $x_0, \ldots, x_n$.
\end{definition}

By Dirichlet's Theorem,  for every $p$-adic number $\xi$, 
we have the relations
\begin{equation} \label{eq:lose}
\lambda_{n}(\xi)\geq \hla_n(\xi) \geq \frac{1}{n}, \qquad\quad w_{n}(\xi)\geq \hw_{n}(\xi)\geq n.
\end{equation}
We remark that for the uniform exponents, and unlike in Definitions~\ref{Def:1.1} 
and~\ref{Def:2.1},  we cannot impose 
that the integers $x_0, \ldots , x_n$ are coprime, 
without losing property \eqref{eq:lose}. 
Theorem \ref{neu} is the $p$-adic analogue of a result
of Schleischitz \cite[Theorem 2.1]{Schl17}.

\begin{theorem} \label{neu}
    Let $m,n$ be positive integers and $\xi$ a
    $p$-adic transcendental number. Then
    \begin{equation}  \label{eq:glm}
    \widehat{\lambda}_{m+n-1}(\xi)\leq \max \left\{ \frac{1}{w_{m}(\xi)},\frac{1}{\widehat{w}_{n}(\xi)}\right\}.
    \end{equation}
\end{theorem}

The case $m=1$ of Theorem \ref{neu} implies the following result:
For any integer $n \ge 1$ and any transcendental $p$-adic number
$\xi$, we have
\begin{equation}
\label{5.1}
\hla_n (\xi) \le \max \Bigl\{ {1 \over n}, {1 \over \lambda_1 (\xi)} \Bigr\}.
\end{equation}
This is the $p$-adic analogue of \cite[Theorem 1.12]{Schl16}.
An alternative proof of \eqref{5.1} can be supplied by adapting to the $p$-adic
setting the arguments given in \cite[Section 5]{BaBu20}.

For $k \ge 1$, by applying Theorem \ref{neu} with $m=n = \lceil k/2 \rceil$ and noticing that
$k \ge 2 \lceil k/2 \rceil - 1$, we get
$$
\widehat{\lambda}_{k} (\xi) \le \widehat{\lambda}_{2 \lceil k/2\rceil -1}(\xi)
\leq \max \left\{ \frac{1}{w_{ \lceil k/2\rceil}(\xi)},\frac{1}{\widehat{w}_{\lceil k/2\rceil}(\xi)}\right\}
\leq \frac{1}{\left\lceil \frac{k}{2}\right\rceil}.
$$
In fact the stronger claim holds that for some $c=c(k,p,\xi)>0$ the system
\begin{equation}  \label{eq:laubesser}
0 < \max\{|x_0|, \ldots , |x_k|\} \le X, \quad \max_{1 \le m \le k} \,
|x_0 \xi^m - x_m|_p \le cX^{-\frac{1}{\lceil k/2 \rceil} - 1},
\end{equation}
has no integral solution for certain arbitrarily large $X$,
thereby establishing the $p$-adic analogue of a result of Laurent \cite{Lau03}
on uniform simultaneous approximation to the first $n$ powers of a real number. The stronger version \eqref{eq:laubesser} is justified
by application of the refined claims
in Theorems~\ref{zwischen},~\ref{emma} below (that 
imply Theorem~\ref{neu}), we skip the details.
Furthermore, by applying Theorem~\ref{neu} with $m=n=k$, we get
\[
\frac{1}{2k-1} \leq \widehat{\lambda}_{2k-1} (\xi) \leq
\max\left\{ \frac{1}{w_{k} (\xi)}, \frac{1}{\widehat{w}_{k} (\xi)} \right\}  = \frac{1}{\widehat{w}_{k} (\xi)}.
\]
Again a stronger claim involving a multiplicative constant,
as in \eqref{eq:laubesser}, can be proved, which is the $p$-adic analogue of~\cite[Theorem~2b]{DaSc69}. We have thus established the following theorem.

\begin{theorem} \label{laupadic}
    For every positive integer $n$ and every $p$-adic transcendental number $\xi$, we have
    \begin{equation}  \label{eq:dsp}
    \widehat{w}_{n}(\xi)\leq 2n-1,
    \end{equation}
    and
    \begin{equation} \label{eq:ciel}
    \widehat{\lambda}_{n}(\xi)\leq \left\lceil \frac{n}{2}\right\rceil^{-1}.
    \end{equation}

\end{theorem}

For even integers $n$, \eqref{eq:ciel} has already been established by Teuli\'e \cite{Teu02}, who
got the slightly weaker upper bound $1 / \lfloor n/2 \rfloor$ for odd integers $n \ge 5$.

By applying a well-known transference theorem \cite{DaSc69,Teu02},
we deduce from \eqref{eq:laubesser} a slight improvement on
Teuli\'e's results on approximation to a $p$-adic transcendental
number by $p$-adic algebraic numbers (resp., integers) of prescribed
degree. As usual, throughout this paper, the height $H(P)$ of an
integer polynomial $P(X)$ is the maximum of the absolute values of
its coefficients and the height $H(\alpha)$ of a $p$-adic algebraic
number $\alpha$ is the height of its minimal defining polynomial
over $\Z$ with coprime coefficients.
Let $\xi$ be a $p$-adic number. Let $n \ge 2$ be an integer. 
Then, there exists a positive constant $c=c(n,p,\xi)$ such that the inequality 
    \begin{equation}  \label{eq:wirsingen}
    |\xi - \alpha|_p  \le c \,  H(\alpha)^{-\lceil\frac{n}{2}\rceil-1}
    \end{equation}
    has infinitely many solutions in $p$-adic algebraic numbers $\alpha$ of degree
    exactly $n$ and also, if $|\xi|_p = 1$, in $p$-adic algebraic integers of degree exactly $n+1$.

\section{Proof of Theorem \ref{Th:2.0}}

We start with a general observation. It is clear that, for every $p$-adic number $\xi$ and 
every non-zero rational number $r$, any of the   
irrationality exponents $w_k, \widehat{w}_k, \lambda_k$ and  
$\widehat{\lambda}_k(\xi)$ takes the same value at the points $\xi$ and $r \xi$. 
Since $\QQ_p = \bigcup_{m\in\ZZ} p^m \ZZ_p$, for all
$A\subset \mathbb{R}$ and $\nu$ in $\{w_k,\widehat{w}_k, \lambda_k,\widehat{\lambda}_k\}$, 
we have
$$
\dim\{\xi\in\QQ_p\;:\; \nu(\xi)\in A\} = \dim\{\xi\in \QQ_p\; :\; |\xi|_p = 1 \hbox{ and } \nu(\xi)\in A\}. 
$$
Therefore, throughout all the proofs, we may assume without loss of generality that $|\xi|_p=1$.

\begin{proof}[Proof of Theorem \ref{Th:2.0}]  
Let $n \ge 2$ be an integer and $\xi$ a $p$-adic number with $\lambda_n (\xi)
> 1$. Let $\lambda$ be a real number with $1 < \lambda < \lambda_n (\xi)$.
Then, there are integers $q, p_1, \ldots , p_n$, with $Q = \max\{|q|, |p_1|,
\ldots , |p_n|\}$ arbitrarily large and $\gcd(q, p_1, \ldots , p_n) = 1$,
such that
$$
|q \xi^j - p_j|_p < Q^{-\lambda - 1}, \quad j=1, \ldots , n.
$$
Since $|\xi|_p = 1$, we observe that $p$ does not divide the product $q p_1 \cdots p_n$.
Note that
$$
|p_{j+1} - p_j  \xi|_p = |p_{j+1} - q \xi^{j+1} - \xi ( p_j - q \xi^j)|_p < Q^{-\lambda - 1}, \quad j=1, \ldots , n-1.
$$
Set $p_0 = q$. Observe that, for $j=1, \ldots , n-1$, we have
$$
\Delta_j :=    
p_{j-1} p_{j+1} - p_j^2
= p_{j-1} (p_{j+1} - p_j  \xi) - p_j (p_j  - p_{j-1} \xi ),
$$
thus, by the triangle inequality,
$$
|\Delta_j |_p < Q^{-1-\lambda}.
$$
Since $|\Delta_j| \le 2 Q^2$, we get that any non-zero $\Delta_j$
should satisfy $|\Delta_j|_p \ge (2 Q^2)^{-1}$. Therefore, if $Q$ is
sufficiently large and $\lambda > 1$, we get that
$$
\Delta_1 = \ldots = \Delta_{n-1} = 0,
$$
which implies that there exist coprime non-zero integers $a, b$ such that
$$
{p_1 \over q} = {p_2 \over p_1} = \ldots = {p_n \over p_{n-1}} = {a \over b}.
$$
We deduce at once that the point
$$
\Bigl( {p_1 \over q} , \ldots , {p_n \over q} \Bigr)
= \Bigl( {a \over b} , \ldots , \Bigl( {a \over b} \Bigr)^n \Bigr)
$$
lies on the Veronese curve $x \mapsto (x, x^2, \ldots , x^n)$ and that
$q$ (resp., $p_n$) is an integer multiple of $b^n$ (resp., of $a^n$).
We remark that, since $\gcd(q,p_{1},\ldots,p_{n})=1$, 
we have in fact $(q,p_{1},\ldots,p_{n})=\pm(b^n,b^{n-1}a,\ldots,a^n)$.
In particular, we get
\begin{equation}
\label{Approx}
|q \xi - p_1 |_p = | b \xi -  a|_p < Q^{-1 - \lambda} \le (\max\{|a|, |b|\})^{-n(1 + \lambda)}.
\end{equation}
This proves that
\begin{equation}
\label{ApproxCor}
\hbox{$\lambda_n (\xi) > 1$ implies $\lambda_1 (\xi) \ge n (1 + \lambda_n (\xi) ) - 1$.}
\end{equation}
Furthemore, since $Q$ (and, thus, $\max\{|a|, |b|\}$) is arbitrarily large, we deduce from \eqref{Approx}
and the (easy half of the) $p$-adic
Jarn\'\i k--Besicovich theorem (see \eqref{1.1}) that
$$
\dim \{\xi \in \Q_p : \lambda_n (\xi) \ge \lambda\} \le {2 \over n(1 + \lambda)}.
$$
The reverse inequality is easier.
Let $\lambda > \lambda_1 (\xi)$ be a real number and $a, b$ be (large) integers not divisible by $p$
(recall that $|\xi|_p = 1$) such that
$$
|b \xi - a|_p \le \max\{|a|, |b|\}^{-1 - \lambda}.
$$
Then, for $j = 1, \ldots , n$, we have
$$
|b^j \xi^j - a^j|_p \le \max\{|a|, |b|\}^{-1 - \lambda}
\quad
\hbox{and}
\quad
|b^n \xi^j - a^j b^{n-j}|_p \le \max\{|a|, |b|\}^{-1 - \lambda},
$$
giving that
$$
\max_{1 \le j \le n} \, |b^n \xi^j - a^j b^{n-j}|_p \le \max\{|a^n|, |b^n|\}^{- (1 + \lambda) /n}.
$$
This proves that
\begin{equation}
\label{ApproxCorBis}
\lambda_n (\xi) \ge \frac{1 + \lambda_1 (\xi)}{n} - 1,
\end{equation}
which is the content of \cite[Lemma 5.1]{BBDO}.
Consequently,
$$
\{\xi \in \Q_p : \lambda_n (\xi) \ge \lambda\} \supset \{\xi \in \Q_p : \lambda_1 (\xi) \ge n (\lambda + 1) - 1\},
$$
and we conclude by (1.1).
\end{proof}

\section{Proofs of the transference inequalities}

Before proceeding with the proof of Theorem \ref{Th:3.1}, we have to adapt to the $p$-adic setting the
argument used by Champagne and Roy \cite{ChamRoy19} to improve
the transference inequality obtained in \cite[Theorem 3.1]{BaBu20}.
We begin with a definition.

\begin{definition}
\label{Def:2.1lead}
Let $n \ge 1$ be an integer and $\xi$ a $p$-adic number.
We denote by $w^{{\rm lead}}_n(\xi)$
the supremum of the real numbers $w$ such that, for
arbitrarily large real numbers $X$, the inequalities
$$
\max\{|x_0|, \ldots , |x_n|\} \le X,  \quad 0<|x_n \xi^n + \ldots + x_1 \xi + x_0|_p  \le X^{-w-1},
$$
have a solution in integers $x_0, \ldots, x_n$ with $|x_n|= \max\{|x_0|, \ldots , |x_n|\}$
and $|x_n|_p = 1$.
\end{definition}

The next lemma is a $p$-adic analogue to \cite[Theorem]{ChamRoy19}.

\begin{lemma} \label{ChampRoy}
Let $k\ge 1$ be an integer and $r_0,\dots,r_k$ be distinct integers.
There is an integer $M\ge 1$ such that, for each nonzero
$\xi$ in $\Q_p$, there exists at least one index $i$ in $\{0,1,\dots,k\}$
with $\xi \neq p r_i |\xi|_p^{-1}$, for which the $p$-adic number
 $\xi_i=1/(M(|\xi|_p \xi - p r_i))$   
satisfies $w_k^{{\rm lead}}(\xi_i)=w_k(\xi_i)=w_k(\xi)$.
\end{lemma}

\begin{proof}

If $\xi = a/b$ is rational and $i$ is such that $\xi \neq p r_i |\xi|_p^{-1}$, then  
$\xi_i$ is rational and $w_k(\xi) = w_{k}(\xi_i)= w_k^{\rm lead}(\xi_i)= 0$.   

Let $\xi$ be an irrational $p$-adic number.
Define $\xi' = |\xi|_p\xi$
and observe that $|\xi'|_p = 1$ and $w_k (\xi) =
w_k (\xi')$. We follow the proof of \cite{ChamRoy19}, with suitable
modifications. By Lagrange Interpolation Formula, there exists a
positive constant $C_1$ such that every integer polynomial $P(X)$
of degree at most $k$ satisfies
$$
H(P) \le C_1\max\{|P(p r_i)|\,:\,0\le i\le k\}.
$$
There also exists a positive constant $C_2$ such that
$$
\max\{ H(P(X+ p r_i)) \,:\,0\le i\le k\} \le C_2 H(P).
$$
Let $M$ be an integer greater than $C_1C_2$ and not divisible by
$p$. By definition of $w_k (\xi')$, there exists a sequence of
polynomials $(P_j)_{j\ge 1}$ in $\Z[X]$ of degree at most $k$ such
that $H(P_{j+1}) > H(P_j)$ for $j \ge 1$ and
\begin{equation}
 \label{eq2}
 \lim_{j\to\infty} -\frac{ \log|P_j(\xi')|_p}{\log H(P_j)} = w_k(\xi) + 1.
\end{equation}
Furthermore, since $|\xi'|_p = 1$, we may assume that, for $j \ge 1$,
the constant coefficient of $P_j (X)$ is
not divisible by $p$. In particular, $P_j (p r_i)$ is not divisible by $p$, for $j \ge 1$ and $0 \le i \le k$.

There exist $i$ in $\{0,1,\dots,k\}$ and an
infinite set ${\mathcal S}$ of positive integers such that $H(P_j) \le C_1 |P_j(p r_i)|$ for every $j$ in ${\mathcal S}$.
Let $j$ be in ${\mathcal S}$. Set
\[
Q_j(X)=(MX)^k P_j\Big(\frac{1}{M X}+ p r_i\Big) \in \Z[X].
\]
The absolute value of the coefficient
of $X^k$ in $Q_j(X)$ is $M^k |P_j(p r_i)|$, while its other coefficients
have absolute value at most
\[
 M^{k-1} H(P_j(X+ p r_i))
  \le C_2M^{k-1} H(P_j)
  \le C_1C_2M^{k-1} |P_j(p r_i)|
  \le M^{k} |P_j(p r_i)|.
\]
Consequently, the absolute value of the leading coefficient of $Q_j
(X)$ is equal to the height of $Q_j(X)$ and is not divisible by $p$.
We also have $|p r_i - \xi'|_p=1$ and, setting
$$
\xi_i= \frac{1}{M(\xi'- p r_i)},
$$
we check that $|\xi_i|_p = 1$ and
\[
 |Q_j(\xi_i)|_p = |M \xi_i|_p^k |\, P_j(\xi')|_p = |P_j(\xi')|_p.
\]
Since the quotient $H(Q_j) / H(P_j)$ is
bounded from above and from below by positive constants, we deduce from \eqref{eq2} that
$$
 \lim_{j\to\infty, j \in S}  - \frac{\log|Q_j(\xi_i)|_p}{\log H(Q_j)} = w_k(\xi) + 1.
$$
This shows that $w_k^{{\rm lead}}(\xi_i)\ge w_k(\xi)$.
As $\xi_i$ is the image
of $\xi'$ by a linear fractional transformation with rational
coefficients, we have
$w_k(\xi_i) =w_k(\xi') = w_k(\xi)$.  Since
$w_k(\xi_i)\ge w_k^{{\rm lead}}(\xi_i)$, this proves the lemma.
\end{proof}

\begin{proof}[Proof of Theorem \ref{Th:3.1}]
The case $k=1$ has been established in \cite{Jar39}.  
Let $k, n$ be integers with $2 \le k \le n$.
Let $\xi$ be a transcendental $p$-adic number.
Let $\eps$ be a positive real number.

Assume first that $w_k (\xi) = w_k^{{\rm lead}}(\xi)$ and that $w_k  (\xi)$ is finite.
For arbitrarily large integers $H$, there exist integers $a_0, a_1, \ldots , a_k$, not all zero,
such that $ H = |a_k| = \max\{ |a_0|, |a_1|, \ldots , |a_k|\}$ and $p$ does not divide $a_k$, and 
\begin{equation}
\label{3.0}
H^{-w_k(\xi) - 1 - \eps} \le  |a_k \xi^k + \ldots + a_1 \xi + a_0|_p \le H^{-w_k(\xi) - 1 + \eps}.
\end{equation}
Take such an integer $H$ and set
$$
\rho := a_k \xi^k + \ldots + a_1 \xi + a_0.
$$
Observe that the absolute value of the determinant of the $(n+1) \times (n+1)$ matrix
$$
M:=\left(\begin{matrix} 1 & 0 &0&\cdots& 0 &\cdots&\cdots&\cdots&0\cr
 0& 1& 0 & &\cdots& 0 & \cdots&\cdots&0\cr
\vdots&\vdots&\vdots&\ddots&\vdots & \vdots &\vdots&&\vdots\cr
0 &0&0&\cdots & 1 & 0 &0&\cdots&0\cr
a_0&a_1&a_2&\cdots & a_{k-1} &a_k&0&\cdots&0\cr
0&a_0&a_1&\cdots & a_{k-2} &a_{k-1}&a_k&\cdots&0\cr
\vdots&\vdots&\vdots&\ddots& \vdots & \vdots&\vdots&\ddots&\vdots\cr
0&0&0&\cdots&\cdots & \cdots &\cdots&a_{k-1}&a_k\\
\end{matrix}\right).
$$
is equal to $H^{n-k+1}$.
It then follows from Satz 1 of Mahler \cite{Mah34} that
there exist integers $v_0, \ldots , v_n$, not all zero, such that
$$
|v_0 \xi^j - v_j|_p \le |\rho|_p, \quad 1 \le j \le k-1,
$$
$$
|a_0 v_i + a_1 v_{i+1} + \ldots + a_k v_{i+k}| \le p^{-1}, \quad 0 \le i \le n-k,
$$
$$
|v_i| \le (p H)^{(n-k+1)/k} |\rho|_p^{-(k-1)/k}, \quad 0 \le i \le k-1.
$$
Since the $a_j$'s and $v_j$'s are integers, we get that
$$
a_0 v_i + a_1 v_{i+1} + \ldots + a_k v_{i+k} = 0, \quad 0 \le i \le n-k.
$$
Taking $i=0$ above and recalling that $|a_k| = \max\{ |a_0|, |a_1|, \ldots , |a_k|\}$,
we get that $|v_k| \le k (pH)^{(n-k+1)/k} |\rho|_p^{-(k-1)/k}$ and,
by increasing $i$ step by step, in a similar way
we estimate
$$
|v_{k+1}|, \ldots , |v_n| \le k^{n-k} (pH)^{(n-k+1)/k} |\rho|_p^{-(k-1)/k}.   
$$
Furthermore, for $i=0, \ldots , n-k$, we have
$$
\begin{array}{rl}
|a_k v_0 \xi^{i+k}  - a_k v_{i+k}|_p &=\displaystyle  \bigl| v_0 \left(
{a_{k-1} \xi^{i+k-1} + \ldots + a_1 \xi^{i+1} + a_0 \xi^i - \rho
\xi^i } \right) \\[2ex]
&\displaystyle  \hskip 18mm  - {(a_0 v_i + a_1 v_{i+1} + \ldots +
a_{k-1} v_{i+k-1} )} \bigr|_p. \\
\end{array}
$$
Since $|a_k|_p = 1$, we derive inductively that
\begin{equation}
\label{3.4}
|v_0 \xi^{i+k} - v_{i+k}|_p  \ll_{p, n, \xi}  |\rho|_p, \quad i=0, \ldots , n-k.
\end{equation}
Recall that
\begin{equation}
\label{3.5}
\max\{|v_0|, |v_1|, \ldots , |v_n|\} \ll_{p, n} H^{(n-k+1)/k} |\rho|_p^{-(k-1)/k}.
\end{equation}
Since $\eps$ can be taken arbitrarily close to $0$,
we deduce at once from \eqref{3.0}, \eqref{3.4}, and \eqref{3.5} that
$$
\lambda_n (\xi) \ge {w_k  (\xi) - n + k \over (k-1) w_k  (\xi) + n}.
$$
An inspection of the proof shows that it yields $\lambda_n (\xi) \ge 1 / (k-1)$
when $w_k  (\xi)$ is infinite, so \eqref{3.1} holds in all cases, provided that $w_k (\xi) = w_k^{{\rm lead}}(\xi)$.

If $w_k (\xi) > w_k^{{\rm lead}}(\xi)$, then we apply \eqref{3.1} with $\xi$ replaced by
a $p$-adic number $\xi_i$ satisfying the conclusion of Lemma \ref{ChampRoy}.
Since $\lambda_n (\xi) = \lambda_n (\xi_i)$, we deduce that \eqref{3.1} holds in all cases.
\end{proof}

\begin{proof}[Proof of Theorem \ref{Th:2.4}]
Write $\lambda_k = \lambda_k (\xi)$.
Assume that $\lambda_k$ is finite (otherwise, it follows from
\eqref{ApproxCor} and \eqref{ApproxCorBis} that $\lambda_{k+1} (\xi)$ is infinite and we are done).
Let $\eps$ be a positive real number.
There exist
integers $q, v_1, \ldots , v_k$,
with $Q = \max\{|q|, |v_1|, \ldots , |v_n|\}$ arbitrarily large, such that
\begin{equation}
\label{3.7}
Q^{-\lambda_k - 1 - \eps} \le  \max\{ |q \xi - v_1|_p, \ldots , | q \xi^k - v_k|_p \}
\le Q^{-\lambda_k - 1 + \eps},
\end{equation}
where $Q = \max\{|q|, |v_1|, \ldots , |v_k|\}$.
Take such an integer $Q$. 
In particular, $q$ is nonzero.
It follows from Siegel's lemma (see \cite[Lemma 2.9.1]{BoGu06}) that
there exist integers $a_0, a_1, \ldots , a_k$, not all zero, such that
$$
a_0 q + a_1 v_1 + \ldots + a_k v_k = 0
$$
and
$$
1 \le H := \max\{ |a_0|, |a_1|, \ldots , |a_k|\}  
\ll_{k} \, Q^{1/k}.
$$
We may assume $a_{k}\neq 0$, otherwise we replace $k$ 
by the largest index $j$ with $a_{j}\neq 0$ in the argument below. 
Then, we derive from \eqref{3.7} that
\begin{equation}\label{3.7a}
\begin{array}{rl}
|q(a_k \xi^k + \ldots + a_1 \xi + a_0) |_p & = \\[1.5ex]
|q(a_k \xi^k + \ldots + a_1 \xi + a_0) - (a_k v_k + \ldots + a_1 v_1 + a_0 q)|_p
& \leq \, Q^{-\lambda_k - 1 + \eps}.   \cr
\end{array}
\end{equation}
Using triangle inequalities, we get from \eqref{3.7} and
\eqref{3.7a} that
\begin{equation}
\label{3.8} \begin{array}{rl}
 &  | a_k q \xi^{k+1}  + a_{k-1} v_k + a_{k-2} v_{k-1}+\ldots+a_1 v_2 + a_0 v_1|_p \\[1.5ex]
& \leq  \max\{| q (a_k \xi^{k+1} +  \ldots + a_0 \xi)|_p,
|q (a_{k-1} \xi^k + \ldots + a_0 \xi) - a_{k-1} v_k -  \ldots  - a_0 v_1|_p\}  \\[1.5ex]
 &\ll_{\xi} \max\{ |q(a_k \xi^k + \ldots + a_1 \xi + a_0)|_p , Q^{-\lambda_k - 1 + \eps} \}    \\[1.5ex]
& \ll_{\xi} Q^{-\lambda_k - 1 + \eps}.   \cr
\end{array}
\end{equation}
Since $a_{k}\neq 0$,
it now follows from
$$
|a_{k-1} v_k + a_{k-2} v_{k-1}+\ldots+a_1 v_2 + a_0 v_1| \ll_{k}  Q^{1 + 1 / k},
$$
\eqref{3.7}, and \eqref{3.8} that
$$
1 + \lambda_{k+1} (\xi) \ge {\lambda_k (\xi)  + 1 - \eps \over 1 + 1/k}.
$$
As $\eps$ can be chosen arbitrarily close to $0$, we deduce that
$$
(k+1) \bigl( 1 + \lambda_{k+1} (\xi) \bigr) \ge k \bigl( 1 + \lambda_k (\xi) \bigr).
$$
This concludes the proof of the first inequality of the theorem. By iterating it, we
immediately get the last inequality. In particular, we obtain
\begin{equation}
\label{TwoIneq}
\hbox{$(k +1) ( \lambda_{k+1} (\xi) + 1) \ge \lambda_1 (\xi) + 1 \ $ and
$\ k   (\lambda_k (\xi) + 1) \ge \lambda_1 (\xi)  + 1$}.
\end{equation}

Assume now that $\lambda_{k+1} (\xi) > 1$. Then, we also have $\lambda_{k} (\xi) > 1$ and we get
from \eqref{ApproxCor} that
$\lambda_1 (\xi) \ge (k+1) (1 + \lambda_{k+1} (\xi) ) - 1$ and
$\lambda_1 (\xi) \ge k (1 + \lambda_k (\xi) ) - 1$. Combined with \eqref{TwoIneq}, this gives at once
$$
(k+1) \bigl( 1 + \lambda_{k+1} (\xi) \bigr) = k \bigl( 1 + \lambda_k (\xi) \bigr).
$$
By iterating this equality, we obtain that the last inequality of the theorem
is an equality when $\lambda_n (\xi)$ exceeds $1$.
\end{proof}

\section{Proof of Theorem \ref{Th:2.3}}

The following result, established by Bernik and Morotskaya \cite{BeMo86,Moro87} (see also
\cite[Section 6.3]{BeDo}), is a key ingredient in the proof of Theorem \ref{Th:2.3}.

\begin{theorem}
\label{Th:Moro}
For every positive integer $n$ and every real number $w$ with $w \ge n$, we have
\begin{equation}
\label{1.2}
\dim \{ \xi \in \Q_p : w_n (\xi) \ge w \} = {n + 1 \over w + 1}.
\end{equation}
\end{theorem}

We observe that Theorem \ref{Th:Moro} extends \eqref{1.1}.

\begin{proof}[Proof of the first assertion of Theorem \ref{Th:2.3}]
Let $k, n$ be integers with $1 \le k \le n$.
For $k=1$, inequality \eqref{2.1} follows from combining \eqref{1.1}
with the inclusion
\[
\{ \xi\in \Q_p: \lambda_{n}(\xi)\geq \lambda \} \supset 
\{ \xi\in \Q_p: \lambda_{1}(\xi)\geq n\lambda+n-1 \}
\]
that follows
from the consequence $\lambda_{n}(\xi)\geq (\lambda_{1}(\xi)-n+1)/n$ of Theorem~\ref{Th:2.4}.
For $k \ge 2$
and $\lambda$ in $[1/n,1/(k-1))$,
inequality~\eqref{3.1}
implies that
\begin{equation}
\label{Mo}
\{\xi\in \Q_p : \lambda_n(\xi)\ge \lambda\} \supset \left\{\xi\in \Q_p :
w_k (\xi) \ge {(\lambda+1)n-k \over
1-\lambda(k-1)}\right\}.
\end{equation}
By Theorem \ref{Th:Moro}, this yields \eqref{2.1}.
Finally, the assertion 
is trivial if $\lambda\geq 1/(k-1)$.
\end{proof}

The proof of the second assertion of Theorem \ref{Th:2.3} requires more work.
We keep the same steps as in the proof of its real analogue.
Instead of appealing to a result of Davenport and Schmidt \cite{DaSc69} as in \cite{BaBu20}, we
make use of its $p$-adic analogue, which was proved by
Teuli\'e \cite[Lemme 3]{Teu02}.

We consider the $(n+1)$-tuples $\vp:= (q,p_1,p_2,\ldots, p_n)$ of
integers which approximate at least one point 
$(\xi,\xi^2,\ldots, \xi^n)$ on the Veronese curve, that is, which
satisfy
\begin{equation}
\label{4.1}
|q\xi^i - p_i|_p \ll_{\xi, n} Q^{-\lambda - 1}, \quad i = 1, \ldots , n, \quad   
\hbox{with $Q = \max\{|q|, |p_1|, \ldots , |p_n|\}$. }
\end{equation}
For convenience, we will often write $p_0$ instead of $q$.

Throughout this section, we extensively make use of matrices of the form
$$
\Delta_{m,k}:= \left( \begin{matrix} p_{k-m+1}&p_{k-m+2}&\cdots
&p_k\cr p_{k-m+2}&p_{k-m+3}&\cdots &p_{k+1}\cr
\vdots&\vdots&\ddots&\vdots\cr p_k&p_{k+1}&\cdots&p_{k+m-1} \\
\end{matrix}
\right).
$$
Observe that $\Delta_{m,k}$ is an $m \times m$ matrix with $p_k$ in its antidiagonal.
The matrices $\Delta_{2,k}$ have been used in the proof of Theorem \ref{Th:2.0}.

\begin{proposition}
\label{Prop:4.1}
Assume that a tuple $\vp = (p_0,\ldots, p_n)$ in $\ZZ^{n+1}$
satisfies \eqref{4.1} for some $p$-adic number $\xi$.  
Then, we have
\begin{equation}
\label{4.2a}
|p_i\xi - p_{i+1}|_p \ll_{\xi, n} Q^{-\lambda - 1}, \quad \hbox{for  $i\in\{0,\ldots, n-1\}$,}
\end{equation}
and
\begin{equation}
\label{4.2b}
|\det (\Delta_{2,i})|_p \ll_{\xi, n} Q^{-\lambda - 1}, \quad \hbox{for  $i\in\{1,\ldots, n-1\}$.}
\end{equation}
Conversely, if an integer tuple $\vp$ in $\ZZ^{n+1}$ with $|p_0|_p , \ldots , |p_n|_p \gg 1$ satisfies~\eqref{4.2b},
then there exists a $p$-adic number $\xi$ for which \eqref{4.1} is true.
\end{proposition}

\begin{proof}
For the first two statements, see the beginning of the proof of Theorem \ref{Th:2.0} and observe that 
the assumption $|\xi|_p = 1$ can be removed.  

For the last statement, consider an integer tuple $\vp$ which satisfies \eqref{4.2b}.
Then, for
$i = 1, \ldots , n-1$,  we have
$$
|p_i^2 -  p_{i-1}  p_{i+1} |_p \ll_{\xi, n} Q^{-1-\lambda}.
$$
Setting $\xi:= p_1/p_0$, these inequalities yield,
for $i$ in $\{ 0,1,\ldots,n-1 \}$,
$$
\left|\xi - {p_{i+1} \over p_i}\right|_p \ll_{\xi, n} Q^{-1-\lambda}, \quad
\hbox{thus} \quad |p_i\xi - p_{i+1}|_p \ll_{\xi, n} Q^{-1 - \lambda}.
$$
Now we use induction on $i$. For $i=0$
and setting $q=p_{0}$,
the statement $|q\xi - p_1|_p \ll_{\xi, n}
Q^{-1 - \lambda}$ follows from the last estimate. Assuming that~\eqref{4.1}
is true
for $i$ in $\{ 0,1,\ldots,n-1 \}$,
we deduce from
$$
|q\xi^{i+1} - p_{i+1}|_p  = |(q\xi^i - p_i)\xi + p_i\xi - p_{i+1}|_p \ll_{\xi, n} Q^{-1 - \lambda},
$$
that it is also true for $i+1$.
\end{proof}

The next proposition is the $p$-adic analogue of
 \cite[Proposition~4.2]{BaBu20}.

\begin{proposition}
\label{Prop:4.2}
Let $\vp$ be in $\ZZ^{n+1}$ which
satisfies~\eqref{4.2b}. 
Let $\xi$ be given by the last assertion of Proposition~\ref{Prop:4.2}.  
Then, for any positive integers $m,k$ with $k-m+1\ge 0$ and $k+m-1\le n$, we
have
$$
|\det (\Delta_{m,k}) |_p \ll_{\xi, n} Q^{ - (m-1) (\lambda + 1)}.
$$
\end{proposition}

\begin{proof}
For integers $i, j$ with $0 \le j < i \le n-1$, observe that $|p_j \xi^{i - j} - p_i |_p$ is,
up to a factor that depends on $\xi$ only,
at most equal to
$$
\max\{|p_j \xi^{i - j } - p_{j+1} \xi^{i - j -1}|_p, |p_{j+1} \xi^{i - j - 1} - p_{j+2} \xi^{i - j - 2}|_p, \ldots , |p_{i-1} \xi  - p_{i} |_p \},
$$
thus, by~\eqref{4.2a}, we have
$$
|p_j \xi^{i - j} - p_i |_p \ll_{\xi, n} Q^{ - (\lambda + 1)}.
$$
To compute the determinant of the matrix $\Delta_{m,k}$, we replace its $h$-th column by
$\xi^{m-h}$ times this column minus the last column. Then, we expand the determinant as a sum of $m!$ products
of $m$ terms. Each of these product is the product of an integer by $m-1$ terms which are $p$-adically
 $\ll_{\xi, n} Q^{ - (\lambda + 1)}$.
This proves the proposition.
\end{proof}

The proof of Proposition \ref{Prop:4.2} can easily be adapted to show the next proposition, which is more
general.

\begin{proposition}
\label{Prop:4.3}
Let $\vp$ be in $\Z^{n+1}$ which
satisfies~\eqref{4.2b}  
and $m$
a positive integer. 
Let $\xi$ be given by the last assertion of Proposition~\ref{Prop:4.2}.   
For $i=0, \ldots , n-m+1$, let
$\vy_i$ denote the vector $(p_i,p_{i+1},\ldots, p_{i+m-1})$.
Then, for any sequence $c_1,c_2,\ldots, c_m$ of integers in
$\{0, \ldots , n-k+1\}$, the determinant $d(c_1,\ldots, c_m)$ of the
$m \times m$ matrix composed of the
vectors $\vy_{c_1}, \vy_{c_2},\ldots, \vy_{c_m}$ satisfies
$$
|d(c_1,\ldots, c_m)|_p \ll_{\xi, n} Q^{-(m-1) (\lambda + 1)}.
$$
\end{proposition}

The next auxiliary result is \cite[Lemme 3]{Teu02}.

\begin{theorem}
\label{Th:DS}
Let $a_0,a_1,\ldots, a_h$ be integers with no
common factor throughout. Assume that, for some non-negative
integers $t,k$ with $k+h-1\le t$ and $t+h\le n$, the integers
$p_k,p_{k+1},\ldots, p_{t+h}$ are related by the recurrence relation
$$
a_0p_i+a_1p_{i+1}+\cdots+ a_{h}p_{i+h} = 0, \quad k\le i\le t.
$$
Let $Z$ be the maximum of the absolute values of all the $h\times h$
determinants formed from any $h$ of the vectors $\vy_i:=(p_i,
p_{i+1}, \ldots, p_{i+h-1})$, $i = k, \ldots , t+1$. Let $p^{-
\rho}$ with $\rho > 0$ be the maximum of the $p$-adic absolute
values of all the $h\times h$ determinants formed from any $h$ of
the vectors $\vy_i:=(p_i, p_{i+1}, \ldots, p_{i+h-1})$, $i = k,
\ldots , t+1$. If $Z$ is non-zero, then
$$
\max\{|a_0|,|a_1|,\ldots,|a_h|\} \ll_{\xi, n} (Z p^{- \rho})^{1/(t-k-h+2)}.
$$
\end{theorem}

\begin{proof}[Proof of the second assertion of Theorem \ref{Th:2.3}]
Let $\lambda> 1/\lfloor(n+1)/2\rfloor$ be a
real number and set $m = 1 + \lfloor 1 / \lambda \rfloor$. Let $\xi$ be a
transcendental $p$-adic number such that $\lambda_n (\xi) \ge \lambda$ and
consider an $(n+1)$-tuple $\vp$ for which~\eqref{4.1} is satisfied and $Q$ is large
enough.

Let $h$ be the smallest non-negative integer number such that the matrix
$$
P_h:=\left(\begin{matrix} p_{0}&p_{1}&\cdots &p_{n-h-1}&p_{n-h}\cr
p_{1}&p_{2}&\cdots &p_{n-h}&p_{n-h+1}\cr
\vdots&\vdots&\ddots&\vdots&\vdots\cr
p_h&p_{h+1}&\cdots&p_{n-1}&p_{n} \cr\end{matrix} \right).
$$
has rank at most $h$. Obviously, $h\le \lceil {n+1 \over 2}\rceil$, because
for $\ell = \lceil {n+1 \over 2}\rceil$ the matrix $P_\ell$ has more rows
than columns and its rank is at most $\ell$. Also, we have $h\ge 1$ since
$\vp$ is not the zero vector. On the other hand, for $q=p_0$ large enough, we
get $h\le m$. Indeed,
consider $m+1$ arbitrary columns of the matrix $P_m$.
By Proposition~\ref{Prop:4.3}, the determinant of the   
integer matrix formed from these columns has $p$-adic absolute value at
most $c Q^{ -  m (\lambda + 1)}$ for some positive constant $c=c(n,\xi)$.
Since all its entries are integers of absolute value at most $Q$, the
$p$-adic absolute value of this determinant is either $0$, or at least
$Q^{- (m+1)}$.
Consequently, since
$\lambda>1/m$, for $Q$ large enough, this determinant is zero. Since
$\lambda> 1/\lfloor(n+1)/2\rfloor$, we have
\begin{equation}
\label{4.3}
h\le m\le \left\lfloor{n+1\over 2}\right\rfloor.
\end{equation}

By construction of the matrix $P_h$, there exist integers $a_0,a_1, \ldots, a_h$ with no
common factor such that
\begin{equation}
\label{4.4}
a_0p_i+a_1p_{i+1}+\cdots+ a_{h}p_{i+h} = 0, \quad 0\le i\le n-h.
\end{equation}
Note that the matrix $P_{h-1}$ has rank $h$ and therefore the value of $Z$,
defined in Theorem~\ref{Th:DS}, is non-zero. Moreover,
with $\rho$ defined as in Theorem~\ref{Th:DS}, 
 Propositions~\ref{Prop:4.2} and \ref{Prop:4.3} imply that
$$
Z p^{-\rho} \ll_{\xi, n} Q^h Q^{-(h-1)(\lambda + 1)}.
$$
From inequality~\eqref{4.3} we have $h-1\le n-h$ and
thus all the assumptions of Theorem~\ref{Th:DS} are satisfied. Applied with $k=0$
and $t = n- h$, it yields
$$
H:=\max\{|a_0|,|a_1|,\ldots, |a_h|\}\le (Z p^{-\rho})^{1/(n-2h+2)}
\ll_{\xi, n} Q^{ {h -(h-1)(\lambda + 1) \over n-2h+2}}.
$$

Consider the relation~\eqref{4.4} for $i=0$ and divide it by $p_0=q$. Then,
the condition~\eqref{4.1} implies that
$$
|a_h\xi^h+a_{h-1}\xi^{h-1}+\ldots+a_0|_p \ll_{\xi, n}
|q|_{p}^{-1} Q^{-1-\lambda} \ll_{\xi, n}
 Q^{-1-\lambda}\ll_{\xi, n} H^{ -
{(1+\lambda)(n-2h+2) \over h -(h-1)(\lambda + 1) }}.
$$
Here, in the second inequality,
we used that $|q|_{p}\gg_{\xi,n} 1$, by~\eqref{4.1}. Indeed, if otherwise
$|q|_{p}< \min\{|\xi|_{p}^{-n},1\}$, then, since $p$ does not divide all $p_{i}$, we get 
$|\xi^{i}q|_{p}<1=|p_{i}|_{p}$ and thus  
$|\xi^{i}q-p_{i}|_{p}= \max\{ |\xi^{i} q|_{p}, |p_{i}|_{p} \}=1$ 
for some $i$, in 
contradiction to~\eqref{4.1}.
See also the proof of Theorem~\ref{Th:2.1}.
For the moment assume $h-(h-1)(1+\lambda)>0$.
By applying Theorem \ref{Th:Moro}, we conclude that
$$
\dim \{ \xi \in \Q_p  : \lambda_n (\xi) \ge \lambda \} \le
\max_{1\le h\le m} \left\{ {(h+1)(h -(h-1)(\lambda + 1)) \over (n-2h+2)(1+\lambda)}\right\},
$$
which gives the expected result.
Finally, in the special case $h-(h-1)(1+\lambda)=0$, we get $h=m$ and 
$\lambda=1/(m-1)$. We can then let $\lambda$ 
tend to $1/(m-1)$ from below and use monotonicity and
a limit argument to derive \eqref{2.2} in this case as well.
\end{proof}

\section{Proof of Theorem~\ref{neu}}

Theorem~\ref{neu} directly
follows from the combination of Theorems~\ref{zwischen}
and~\ref{emma} below.

For integers $\ell, n$ with $n\geq 2$ and $1\leq \ell\leq n+1$, we write
$w_{n,\ell}(\xi)$ for the supremum of $w$ for which there are arbitrarily large integers $H$ such that the system
\[
0<|a_{0} + a_{1} \xi + \cdots + a_{n} \xi^{n} |_{p} \leq H^{-w-1},
\qquad \max_{0 \le i \le n} |a_{i}| \le H,
\]
has $\ell$ linearly independent solutions in integer $(n+1)$-tuples $(a_0, a_1, \ldots , a_n)$.
Note that $w_{n,1}(\xi)=w_{n}(\xi)$.
The following theorem is the $p$-adic
analogue of a result obtained in the course of the proof of \cite[Theorem~2.1]{Schl17}.
Its proof is very similar.

\begin{theorem} \label{zwischen}
    Let $m,n$ be positive integers and $\xi$ in $\mathbb{Q}_{p}$. We have
    \[
    w_{m+n-1,m+n}(\xi) \geq \min\{ w_{m}(\xi),\widehat{w}_{n}(\xi) \}.
    \]
    If $m=n=\widehat{w}_{n}(\xi)$, then the following slightly stronger claim holds: For arbitrarily large $H$,
    there exist
    $m+n=2n$
    linearly independent solutions to
    $$
    |P(\xi)|_{p} \leq c(\xi, n) \, H^{- n - 1},
    $$
    in integer polynomials $P$ of degree at most $m+n-1=2n-1$ and height at most $H$,
    where $c(\xi, n)$ is a suitable positive number depending only on $n$ and $\xi$.
\end{theorem}

The second auxiliary result is a transference theorem.

\begin{theorem} \label{emma}
    Let $n\geq 1$ be an integer and $\xi$ in $\mathbb{Q}_{p}$. We have
    \begin{equation}  \label{eq:geleich}
    \widehat{\lambda}_{n}(\xi) \leq \frac{1}{w_{n,n+1}(\xi)}.
    \end{equation}
    More precisely, if for some positive real number $c_1$ and 
    for arbitrarily large integers $H$, the system   
    \begin{equation} \label{eq:testy}
    H(P)\leq H, \quad \vert P(\xi)\vert_{p} \leq c_1 H^{-w-1}
    \end{equation}
    has $n+1$ linearly independent solutions in integer polynomials
    $P$ of degree at most $n$, then, there are a positive constant $c_2$ 
    and arbitrarily large integers $Z$ such that the system  
    \begin{equation}  \label{eq:teu}
    1\leq \max_{0\leq j\leq n}  |y_{j}|  \leq Z, \quad
    \max_{1\leq j\leq n} |y_{0}\xi^{j}-y_{j}|_{p} \leq c_2 Z^{-\frac{1}{w}-1 }
    \end{equation}
    has no solution in integers $y_0, y_1, \ldots , y_n$.
\end{theorem}

As in the real case, there is in fact equality in \eqref{eq:geleich}; the reverse
inequality has been proved in the course of the proof of \cite[Th\'eor\`eme~3]{Teu02}.

Recall that Gelfond's Lemma (see e.g. \cite[Lemma A.3]{BuLiv})
asserts that there exists a real number $K(n)>1$, depending only on the
positive integer $n$, such that all integer polynomials $P,Q$ of degree at most $n$ satisfy
\begin{equation} \label{eq:hoehe}
\frac{1}{K(n)} H(P)H(Q) \leq H(PQ) \leq K(n) H(P)H(Q).
\end{equation}

\begin{proof}[Proof of Theorem~\ref{zwischen}]
    Let $\ell$ be the smallest integer in $\{1,2,\ldots,m\}$ such that $w_{\ell}(\xi)=w_{m}(\xi)$.
    We will show that
    \begin{equation} \label{eq:nonedda}
    w_{\ell+n-1, \ell+n}(\xi)\geq \min\{ w_{\ell}(\xi),\widehat{w}_{n}(\xi)\}=\min\{ w_{m}(\xi),\widehat{w}_{n}(\xi)\}.
    \end{equation}
    Then, the claim follows since, if $\ell<m$, it is sufficient to
    consider products of the involved polynomials by the powers 
    $1,X,\ldots,X^{m-l}$ of their variable $X$.

    Assume first that $w_{\ell}(\xi)$ is finite.
    Let $\varepsilon$ be in $(0,1)$. By Gelfond's Lemma and the
    definition of $\ell$, for arbitrarily large $H$, there exist
    {\em irreducible} polynomials $P (X)$ of degree exactly $\ell$ that satisfy
    \begin{equation} \label{eq:numerouno}
    H(P)=H, \quad \vert P(\xi)\vert_{p} \leq H^{-w_{\ell}(\xi)-1+\varepsilon}.
    \end{equation}
    With $K(n)$ as in \eqref{eq:hoehe}, set
    \[
    {\widetilde H}=\frac{H}{2K(n)}.
    \]
    Setting $w = \min\{\widehat{w}_{n} (\xi), w_{\ell}(\xi)\}$, by definition of $\widehat{w}_{n} (\xi)$,
    there exists a polynomial $Q_0$ in $\mathbb{Z}[X]$ of degree at most $n$ such that
    \[
    H(Q_0)\leq {\widetilde H}, \quad
    0 < \vert Q_0 (\xi)\vert_{p} \leq {\widetilde H}^{- w -1+\varepsilon}.
    \]
    Set $Q(X) = X^{n-d} Q_0 (X)$, where $d$ denotes the degree of $Q_0$.  
    Since $w$ is finite and
    $\varepsilon<1$, we have
    \begin{equation} \label{eq:numerodos}
    H(Q)\leq {\widetilde H}, \quad
    \vert Q(\xi)\vert_{p}
    \leq K_{1}\cdot H^{- w -1+\varepsilon},
    \end{equation}
    for some positive $K_{1}$.

    By \eqref{eq:hoehe} the polynomial $Q_0$ (and, thus, $Q$) cannot be a multiple of $P$. 
    Hence, since $P$ is irreducible, the polynomials $P$ and $Q$ have no common factor over $\mathbb{Z}[X]$.
    Moreover, the pair
    $(P,Q)$ satisfies the properties
    \[
    \max \{ H(P),H(Q)\} =H, \quad
    \max \{ \vert P(\xi)\vert_p, \vert Q(\xi)\vert_p \}\leq K_{1}\cdot H^{- w -1+\varepsilon}.
    \]
    Consider the set $\mathscr{P}$ of $\ell+n$ polynomials
    \begin{equation} \label{eq:p}
    \mathscr{P}=\{ P, XP, X^{2}P,\ldots, X^{n-1}P, Q, XQ,\ldots,X^{\ell-1}Q\}.
    \end{equation}
    The determinant of the $(\ell+n)\times (\ell+n)$ matrix whose rows are given by the coefficients of the polynomials
    in $\mathscr{P}$ is the resultant of $P$ and $Q$, which is non-zero since
    $P$ and $Q$ have no common factor. Hence,
    the elements of $\mathscr{P}$ are linearly independent and span
    the space of polynomials of degree at most $\ell+n-1$.
    Furthermore, it follows from \eqref{eq:numerouno} and \eqref{eq:numerodos} that any $R$
    in $\mathscr{P}$ satisfies
    \begin{equation} \label{eq:oppen}
    H(R)\leq H, \quad \vert R(\xi)\vert_p  \ll_{n,\xi, p}  H^{-\min\{\widehat{w}_{n}(\xi), w_\ell(\xi)\} -1+\varepsilon}
    \end{equation}
    This shows \eqref{eq:nonedda}. If $w_{\ell}(\xi)$ is infinite, then the same type of argument gives \eqref{eq:nonedda},
    we omit the details.

    Finally, if $m=n=\widehat{w}_{n}(\xi)$, in place
    of \eqref{eq:numerouno}
    we may use a variant of Dirichlet's Theorem stating 
    that for every $\xi\in \mathbb{Q}_{p}$ and every large $H$ the system 
    $$
    H(P)\le H, \quad |P(\xi)|_{p}\ll_{n,\xi} H^{-n-1}  
    $$
    has a solution
    in non-zero integer polynomials $P$ of degree at most $n$.
    Consequently, again by Gelfond's Lemma, the real number $\varepsilon$ in \eqref{eq:numerouno}
    can be taken equal to $0$ upon introduction of a multiplicative
    constant. Thus from the method above
    we get the last assertion of the theorem.
\end{proof}

\begin{definition}
    For a convex body $K\subseteq \mathbb{R}^{n+1}$ and a lattice
    $\Lambda\subseteq \mathbb{R}^{n+1}$, for any $j = 1, \ldots , n+1$, the $j$-th of the successive minima
    of $K$ with respect to $\Lambda$, denoted
    by $\tau_{j}(K,\Lambda)$, is the smallest $\lambda$ such that the convex body
    $\lambda K=\{\lambda k: k\in K\}$ contains $j$ linearly independent points of $\Lambda$.
\end{definition}

\begin{proof}[Proof of Theorem~\ref{emma}]
    We follow the initial steps of the proof of
    \cite[Th\'eor\`eme~3]{Teu02}, where we
    apply the shift $n\to n+1$ in the notation.

    Let $\lambda$ be a real number with
    $$
    1>\lambda > \frac{1}{1+\widehat{\lambda}_{n}(\xi)}.
    $$
    By definition of $\widehat{\lambda}_{n}(\xi)$, the system
    \begin{align}
    \max_{0\leq i\leq n} |z_{i}| &\leq p^{(n+1){\lambda} h}, \nonumber \\
    \max_{1\leq j\leq n} |z_{0}\xi^{j}-z_{j}|_{p} &\leq p^{-(n+1)h} \label{eq:t2}
    \end{align}
    has a solution in non-zero integer vectors $(z_{0},\ldots,z_{n})$
    for every large enough $h$.
    Equivalently, for every sufficiently large
    positive integer $h$ the system
    \begin{align}
    \max_{0\leq i\leq n} |y_{i}| &\leq p^{-(n+1)(1-{\lambda})h}, \nonumber \\
    \max_{1\leq j\leq n} |y_{0}\xi^{j}-y_{j}|_{p} &\leq 1 \label{eq:t2bis}
    \end{align}
    has a non-zero solution $(y_{0},\ldots,y_{n})$
    in the $\ZZ$-module $p^{-(n+1)h}\mathbb{Z}^{n+1}$.
    We may assume that $|\xi|_{p}=1$.
    The solutions to \eqref{eq:t2bis}
    in $p^{-(n+1)h}\mathbb{Z}^{n+1}$
    form a lattice generated by
    \begin{align*}
    &( p^{-h(n+1)} , p^{-h(n+1)}b_{1} , \ldots, p^{-h(n+1)} b_{n} )  \\
    &(0,1,0,0\ldots,0),  \\
    &(0,0,1,0,0\ldots,0),  \\
    &\vdots \;\; \vdots \;\; \vdots \;\; \vdots \;\; \vdots \;\; \vdots \;\; \vdots \;\; \vdots \;\; \vdots \\
    &(0,0,0,0,\ldots,1),
    \end{align*}
    for a given set of integers $b_{1},\ldots,b_{n}$. Denote this lattice
    by $\Lambda(h)$. Its dual lattice $\Lambda^{\ast}(h)$, 
    defined by
    \[
    \Lambda^{\ast}(h)= \{ y\in\mathbb{R}^{n+1} : \forall x\in \Lambda,\;
    x\cdot y \in\mathbb{Z} \},
    \]
    where $\cdot$ is the usual scalar product in $\mathbb{R}^{n+1}$, is generated by the vectors
    \begin{align*}
    &( p^{h(n+1)} , 0 , 0, \ldots, 0 ),  \\
    &(-b_{1},1,0,0\ldots,0),  \\
   &\vdots \;\; \vdots \;\; \vdots \;\; \vdots \;\; \vdots \;\; \vdots \;\; \vdots \;\; \vdots \;\; \vdots \\
    &(-b_{n},0,0,0,\ldots,1).  \\
    \end{align*}
    In other words, $\Lambda^{\ast}(h)$ is the sublattice of $\Z^{n+1}$ consisting of the vectors
    $(x_{0},\ldots,x_{n})$ in $\mathbb{Z}^{n+1}$ that satisfy
    \[
    |x_{0}+x_{1}\xi+\cdots+x_{n}\xi^{n}|_{p} \leq p^{-h(n+1)}.
    \]
    Let $K$ be the convex body formed by the vectors $(x_{0},\ldots,x_{n})$ in $\mathbb{R}^{n+1}$ that satisfy
    $\max_{0\leq i\leq n} |x_{i}| \leq 1$, and $K^{\ast}$ its dual convex body
    defined by
    $$
    K^{\ast} := \{ y\in\mathbb{R}^{n+1}: \forall x\in K,\;
    |x\cdot y| \leq 1 \}.
    $$
    Note that $K$ is almost equal to $K^{\ast}$, more precisely
    $c^{-1}K^{\ast}\subseteq K \subseteq cK^{\ast}$ for some $c=c(n)$. 
    Consequently, for $j = 1, \ldots , n+1$, the $j$-th  of the successive minima 
    $\tau_{j}(K^{\ast},\Lambda^{\ast}(h))$ and $\tau_{j}(K,\Lambda^{\ast}(h))$ satisfy
    $$
    \tau_{j}(K^{\ast},\Lambda^{\ast}(h))\asymp \tau_{j}(K,\Lambda^{\ast}(h)).
    $$
    Consider the first of the successive minima $\tau_{1}(K,\Lambda(h))$ of $K$
    with respect to $\Lambda(h)$ and the 
    $(n+1)$-th successive minima $\tau_{n+1}^{\ast}(K^{\ast},\Lambda^{\ast}(h))$
    of $K^{\ast}$ with respect to $\Lambda^{\ast}(h)$. Then, by
    \cite[Theorem VI, p. 219]{Cas59}, we have
    \[
    \tau_{1}(K,\Lambda (h) )\cdot \tau_{n+1}(K^{\ast},\Lambda^{\ast}(h)) \gg 1.
    \]
    We remark that Teuli\'e \cite[(2.5)]{Teu02} used the reverse estimate (which is true as well).
    Since the system \eqref{eq:t2bis} has a non-zero solution in the $\ZZ$-module $p^{-(n+1)h}\mathbb{Z}^{n+1}$
    for every sufficiently large integer $h$, there exists an integer $h_0$ such that
    \[
    \tau_{1}(K,\Lambda(h))\leq p^{-(n+1)(1-{\lambda})h}, \quad h\geq h_{0}.
    \]
    Hence,
     recalling that $c^{-1}K\subseteq K \subseteq cK^{\ast}$,
     we get
    \[
     \tau_{n+1}(K,\Lambda^{\ast}(h)) \gg \tau_{n+1}(K^{\ast},\Lambda^{\ast}(h))
     \gg p^{(n+1)(1-{\lambda})h}, \quad h\geq h_{0}.
    \]
    Consequently, there exists $c > 0$ such that, for $h > h_0$, the system
    \[
    0 < \max\{|x_0|, \ldots , |x_n|\} \le c p^{(n+1)(1-{\lambda})h}, \quad
|x_{0}+x_{1}\xi+\cdots+x_{n}\xi^{n}|_{p} \leq p^{-h(n+1)},
    \]
does not have $n+1$ linearly independent integer vector solutions $(x_{0},\ldots,x_{n})$.   We have shown that
    \[
    w_{n,n+1}(\xi) \leq \frac{1}{1-{\lambda}  }-1= \frac{{\lambda} }{1-{\lambda} }.
    \]
    Since ${\lambda} $ can be chosen arbitrarily close to $(1+\widehat{\lambda}_{n}(\xi))^{-1}$,
    we see that the right
    hand side can be arbitrarily close to $\widehat{\lambda}_{n}(\xi)^{-1}$. This completes the proof of \eqref{eq:geleich}. Analyzing the proof,
    a very similar argument upon taking the contrapositive of the claim 
    gives the implication of \eqref{eq:teu}  from \eqref{eq:testy}. We omit
    the details.
\end{proof}

\bigskip
\noindent
{\bf Acknowledgement}.
The authors are very grateful to the referee for an extremely careful reading and many 
corrections.

\end{document}